\documentclass[11pt]{article}
\usepackage[english]{babel}  
\usepackage[utf8]{inputenc} 

\usepackage{amsmath}
\usepackage{amssymb}
\usepackage{latexsym}
\usepackage{amsthm, amstext}
\usepackage{array, amsfonts, mathrsfs}
\usepackage{graphicx}

\usepackage{enumitem}

\usepackage[numbers,sort&compress]{natbib}\bibpunct[, ]{[}{]}{,}{n}{,}{,}
\theoremstyle{plain}
\newtheorem{theorem}{Theorem}
\newtheorem{lemma}{Lemma}

\newcommand{\re}[1]{(\ref{#1})}

\begin{document}

\title{On the Cauchy problem for the wave equation \\ with data on the boundary}

\author{
M.~N. Demchenko\footnote{St.~Petersburg Department of
V.\,A.~Steklov Institute of Mathematics of
the Russian Academy of Sciences, 
27 Fontanka, St.~Petersburg, Russia.
The research was supported by the RFBR grant 17-01-00529-a. }}

\date{}

\maketitle

\begin{abstract}
We consider the Cauchy problem for the wave equation
in $\Omega\times{\mathbb R}$
with data given on some part of the boundary $\partial\Omega\times{\mathbb R}$.
We provide a reconstruction algorithm for this problem
based on analytic expressions.
Our result is applicable to the problem of determining nonstationary wave field arising in
geophysics, photoacoustic tomography, tsunami wave source recovery.

\smallskip

\end{abstract}

\medskip

\section{Introduction}\label{intro}
Consider the wave equation
\begin{equation}
  \partial^2_t u - \Delta u = 0  \label{waveq}
\end{equation}
in a space-time cylinder $\Omega\times{\mathbb R}$
($\Omega$ is a domain in ${\mathbb R}^n$, 
$n\geqslant 2$,
$\Delta$ is the Laplace operator in ${\mathbb R}^n$).
We study the problem of determining a solution $u$
from Cauchy data $u$, $\partial_\nu u$
($\nu$ is the outward unit normal to $\partial\Omega$) 
given on some part of the boundary
$\partial\Omega\times{\mathbb R}$.
In contrast to the classical Cauchy problem for the
wave equation with the data on a space-like surface,
the problem in consideration is ill-posed~\cite{is, LRSh}.
However, solution $u$ is uniquely determined
in some part of $\Omega\times{\mathbb R}$ depending on the set,
on which the Cauchy data are given.
This can be inferred from the unique continuation property
for the wave equation,
which is provided
by Holmgren's theorem
(or Tataru's theorem~\cite{Tat} in case of a hyperbolic equation with variable coefficients).

We will denote points in ${\mathbb R}^n$ by $(x,y)$, where
$x=(x_1,\ldots,x_{n-1})\in{\mathbb R}^{n-1}$, $y\in{\mathbb R}$.
We will consider the case when $\Omega$ is a subgraph
of a $C^\infty$-smooth function $Y(x)$ that satisfies a 
certain growth condition:
\begin{equation}\label{subgraph}
  \Omega = \{(x,y)\,|\, x\in{\mathbb R}^{n-1},\,\, -\infty < y < Y(x) \},
  \quad |Y(x)| \leqslant C_1 + C_2 |x|, \,\, C_2<1
\end{equation}
(see fig.~\ref{figcone}). 
\begin{figure}[b!]\centering
  \includegraphics[width=.5\textwidth]{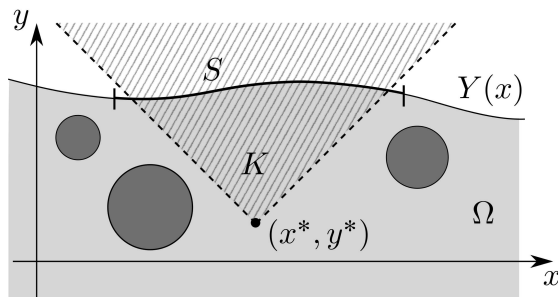}
  \caption{The domain $\Omega$ is a subgraph
of the function $Y(x)$. The darkened region in $\Omega$ is the set $\omega$ containing scatterers and inhomogeneities.
The hatched region is the cone $K$.}
  \label{figcone}
\end{figure}
In case $Y(x)\equiv{\rm const}$ the domain $\Omega$ is a half-space.
Fix a point $(x^*,y^*)\in \Omega$ and
$t^*\in{\mathbb R}$.
We will obtain an algorithm (formula~\re{formula}, 
sec.~\ref{determ}),
which allows determining $u(x^*,y^*,t^*)$
from the Cauchy data on the set
\begin{equation}\label{Cauchy_set}
  \{ (x,y,t)\,|\, (x,y)\in S,\,\, T_-(x)\leqslant t\leqslant T_+(x)\}.
\end{equation}
Here $T_\pm(x) = t^*\pm(Y(x)-y^*)$,
$S$ is a bounded relatively open subset of $\partial\Omega$ that contains the intersection
$K\cap \partial\Omega$, where
\begin{equation}
  K = \{(x,y)\in{\mathbb R}^n\,|\,\, y-y^* \geqslant |x-x^*|\}
  \label{cone}
\end{equation}
(note that the intersection $K\cap\partial\Omega$ is itself bounded
due to the growth condition in~\re{subgraph}).
The cone $K$ and the set $S$ are shown on fig.~\ref{figcone}.
Evidently, $Y(x) > y^*$ whenever $(x,y)\in K\cap\partial\Omega$.
We assume that $S$ is chosen in such a way that the same inequality
holds true whenever $(x,y)$ belongs to the closure $\overline S$, which leads to the inequality 
$T_-(x) < T_+(x)$. Thus the condition on $t$ in~\re{Cauchy_set} makes sense.

Note that formula~\re{formula} allows determining $u(x^*,y^*,t^*)$ from Cauchy data on various subsets of the boundary.
Indeed, one can take a cone $K$ 
with a different orientation than that in~\re{cone} 
(though its vertex should be in $(x^*,y^*)$),
which yields a different set $S\subset\partial\Omega$.
Next we may choose a Cartesian coordinate system in ${\mathbb R}^n$, in which $K$ is represented by~\re{cone}, and then apply formula~\re{formula}.
However, it is required that the domain $\Omega$ can be
represented by~\re{subgraph} in chosen coordinates.
Recall that in Cauchy problems for elliptic equations
and in the analytic continuation problem,
the solution is also uniquely determined from Cauchy data 
on various sets.

The wave equation~\re{waveq} describes wave processes
of various nature in homogeneous media.
Thus our result can be used for determination
of the wave field in a domain from boundary measurements. 
It will be shown in the end of sec.~\ref{determ}
that our algorithm is applicable also in case when 
the solution
$u(x,y,t)$ is defined for $(x,y)$ belonging to
some subset of $\Omega$.
This may correspond, for example, to
the wave process
in the medium that occupies the domain $\Omega$ and contains inhomogeneities and scatterers
located in some set $\omega$
(see fig.~\ref{figcone}).
Under the condition $\omega\cap K=\emptyset$,
formula~\re{formula} can be used to find the wave field
in homogeneous part of the medium (i.e. in $\Omega\setminus\omega$)
without knowledge
of structure of inhomogeneities and scatterers.

The problem of determination of nonstationary wave field
arises in various applications, such as
geophysics~\cite{Kab}, photoacoustic tomography~\cite{Kruger, Nat}, tsunami wave source recovery~\cite{vorchev},
and coefficient inverse problems~\cite{bel, mn3, mn5}.

The Cauchy problem for hyperbolic equations with data
on the boundary has been extensively studied. 
We refer the reader to monographs~\cite{is, LRSh}
containing overview of related results.
Most of these results are Carleman type estimates,
which provide uniqueness of solution 
and conditional stability estimates.
In case of the wave equation with constant coefficients,
reconstruction algorithms based on analytic expressions were obtained.
However, most of them require the nonlocal Cauchy data,
which means that the surface $S$ coincides with the entire boundary $\partial\Omega$~\cite{FPR, FHR, Nat, Quintoetal, Blag}.
As to the case of local data, we mention
the famous result of R.~Courant
on the Cauchy problem for ultrahyperbolic equation
in a half-space
(see~\cite{K}),
and papers~\cite{Pal, mn6}, in which two-dimensional domains were considered.

The author is thankful to prof.~M.~I.~Belishev and prof.~A.~P.~Kiselev for helpful discussions.

\section{Special solution of Laplace equation}
The proof of our main result 
will be based on some special solution of the wave equation,
which depends on a small parameter $h$ and enjoys a certain
localization property as $h\to 0$.
To construct such solution, we will need a special
solution of Laplace equation, which is the subject
of this section.

We will use the notation $\Delta_x = \sum_j \partial^2_{x_j}$
(thus for Laplace operator $\Delta$ in ${\mathbb R}^n$ we have
$\Delta = \Delta_x + \partial^2_{y}$).
Consider the following Cauchy problem for Laplace equation in ${\mathbb R}^n$:
\begin{gather} 
  \Delta \varphi = 0, \label{helmholtz}\\
  \varphi|_{y=0} = \frac{e^{-x^2/h}}{(\pi h)^m}, \quad
  \partial_y \varphi|_{y=0} = 0, \label{helmholtz_cauchy}
\end{gather}
where $h>0$, $m=(n-1)/2$, and
for real or complex vector
$\zeta= (\zeta_1,\ldots,\zeta_{n-1})$ we put
$\zeta^2 = \sum_j \zeta_j^2$.
Note that $\varphi(x,0)$ is the Gaussian distribution in ${\mathbb R}^{n-1}$.

\begin{lemma}\label{ell_lemma}
For any $h>0$ there is a unique
$C^\infty$-smooth function $\varphi(x,y)$ in ${\mathbb R}^n$
that satisfies equation~\re{helmholtz} and initial
conditions~\re{helmholtz_cauchy}.
Besides, 
\begin{equation}
  \varphi, \partial_{x,y} \varphi, \partial^2_{x,y} \varphi \to 0,
  \quad h\to 0, \quad
  \textrm{if } |x|>|y|;
  \label{vphilim}
\end{equation}
the convergence is uniform on every compact set in
$\{|x|>|y|\}$.
\end{lemma}
\begin{proof}
We will study Cauchy problem~\re{helmholtz}, \re{helmholtz_cauchy}, using a method, which is similar 
to that described in~\cite{Hor} (Lemma 9.1.4),
where Laplace equation is reduced to the wave equation.
This method provides a representation of the solution
$\varphi(x,y)$ in terms of the fundamental solution $G(x,y)$ 
of the wave equation.
The latter is defined by the following equalities:
\begin{gather}
  \partial^2_{y} G - \Delta_x G = 0, \label{fund1}\\
  G|_{y=0} = 0, \quad \partial_y G|_{y=0} = \delta(x) \label{fund2}
\end{gather}
($\delta$ is the Dirac delta function).
For any $y$, functions $\partial_y^\beta G(\cdot,y)$, 
$\beta\geqslant 0$,
are compactly supported distributions 
supported in the ball $\{x \,|\, |x|\leqslant |y|\}$.
The Fourier transform of $\partial_yG(\cdot,y)$ 
is equal to $\cos(y |\xi|)$.
Put
\begin{equation} \label{vphi}
  \varphi(x,y) = \frac{1}{(\pi h)^m}
  \left\langle \partial_y G(x',y), e^{-(x - i x')^2/h}\right\rangle .
\end{equation}
Here and further in this proof, angle brackets mean 
pairing of a distribution in variable $x'$ with
a test function.
Although the smooth function $e^{-(x - i x')^2/h}$ 
is not compactly supported in $x'$, 
the right hand side makes sense as the distribution 
$\partial_y G(\cdot,y)$ is compactly supported.
Clearly the derivatives of $\varphi$ can be represented in a similar way
\begin{equation} \label{dvphi}
  \partial^\alpha_x \partial^\beta_y \varphi(x,y) = \frac{1}{(\pi h)^m}
  \left\langle \partial^{\beta+1}_y G(x',y), 
  \partial^\alpha_x e^{-(x - i x')^2/h}\right\rangle .
\end{equation}
This representation and the fact that 
$\partial^2_{y} G(x,0) = 0$ implies the second equality 
in~\re{helmholtz_cauchy}. The first equality 
in~\re{helmholtz_cauchy} follows from 
the definition~\re{vphi} and the second condition 
in~\re{fund2}.

The function $\varphi(x,y)$ and its derivatives with respect 
to $x, y$ can be considered as analytic functions in complex variables $x_1, \ldots, x_{n-1}$, which follows from~\re{vphi}
and~\re{dvphi}. 
Function
\begin{align*}
  &\psi(x,y) = \varphi(i x,y)
  =\frac{1}{(\pi h)^m}
  \left\langle \partial_y G(x',y), e^{(x - x')^2/h}\right\rangle 
\end{align*}
is a convolution of $(\pi h)^{-m} e^{x^2/h}$ with 
$\partial_y G(x,y)$ with respect to $x$ and thus due to~\re{fund1} satisfies
the wave equation
\[
  \partial^2_{y} \psi - \Delta_x \psi = 0
\]
for all $(x,y)\in{\mathbb R}^n$.
For any fixed real $y$ the left hand side here is
an analytic function in $x_1,\ldots, x_{n-1}$.
Therefore, this equation is satisfied for all complex
$x_1, \ldots, x_{n-1}$, which yields
\begin{align*}
  (\partial^2_{y} \varphi + \Delta_x \varphi)(x,y)
  =(\partial^2_{y} \psi - \Delta_x \psi)(-ix,y) = 0.
\end{align*}
Thus equation~\re{helmholtz} is satisfied.
As is well known the solution $\varphi$
of problem~\re{helmholtz}, \re{helmholtz_cauchy} is unique.

Now turn to assertion~\re{vphilim}.
Since the distribution $\partial_y^\beta G(\cdot,y)$ is
supported in the ball $\{x \,|\, |x|\leqslant |y|\}$,
its pairing with a test function
$f$ is determined by restriction of $f$ to
any neighborhood of the specified ball.
Applying Fourier transform, one can easily obtain
the following estimate
\begin{equation}\label{Gf}
  \left|\left\langle \partial_y^\beta G(x',y), f(x')\right\rangle \right|
  \leqslant C_{y,\varepsilon} \max_{\scriptsize\begin{array}{c} |x'|^2\leqslant |y|^2+\varepsilon,\\ |\alpha|\leqslant n+\beta-1\end{array}} |\partial^\alpha f(x')|,
\end{equation}
where $\varepsilon>0$, and constant $C_{y,\varepsilon}$ is bounded whenever $y$ 
is bounded.
Suppose $|x| > |y|$.
Put $f(x') = (\pi h)^{-m} e^{-(x - i x')^2/h}$, 
$\varepsilon = (|x|^2 - |y|^2)/2$.
For $|x'|^2 \leqslant |y|^2+\varepsilon$ we have
\[
  |e^{-(x - i x')^2/h}|/h^m = e^{(x'^2 - x^2)/h}/h^m 
  \leqslant e^{(y^2 - x^2)/(2h)}/h^m \to 0, \quad h\to 0.
\]
The derivatives $\partial_{x'}^\alpha f(x')$ 
are treated the same way.
In view of~\re{vphi}, \re{Gf}, we obtain~\re{vphilim}
for the function $\varphi$.
The same assertion for the derivatives of $\varphi$
is proved by a similar argument using~\re{dvphi}
instead of~\re{vphi}.
\end{proof}

\section{Special solution of the wave equation}
\label{spec_wave}
The following lemma provides transformation of a solution
of Laplace equation to a solution 
of the wave equation~\re{waveq}.
\begin{lemma}\label{wave_lemma}
Suppose a $C^\infty$-smooth function $\varphi(x,y)$ in ${\mathbb R}^n$
satisfies the following relations
\begin{equation}\label{vphiprobl}
  \Delta \varphi = 0, \quad \partial_y\varphi|_{y=0} = 0.
\end{equation}
Then the function
\begin{equation}\label{wdef}
  w(x,y,t) = \frac{1}{\pi} \int_0^{\pi/2}
  \varphi\left(x, \sqrt{y^2 - t^2} \cdot \sin s\right) \,ds
\end{equation}
is $C^\infty$-smooth in the set
$x\in{\mathbb R}^{n-1}$, $t\in{\mathbb R}$, $y > |t|$, and satisfies the wave equation
\begin{gather}
  \partial^2_{t} w - \Delta w = 0. \label{w1} 
\end{gather}
\end{lemma}
\begin{proof}
We have
\begin{align*}&\pi \partial_y w = \frac{y}{\sqrt{y^2-t^2}} \int_0^{\pi/2} 
  (\partial_y \varphi)\left(x, \sqrt{y^2 - t^2} \cdot \sin s\right) \sin s\, ds,
\end{align*}
\begin{align*}&\pi \partial_t w = \frac{-t}{\sqrt{y^2-t^2}} \int_0^{\pi/2} 
  (\partial_y \varphi)\left(x, \sqrt{y^2 - t^2} \cdot \sin s\right) \sin s\, ds.
\end{align*}
Next
\begin{align*}
  &\pi \partial^2_{y} w = \frac{-t^2}{(y^2-t^2)^{3/2}} \int_0^{\pi/2} 
  (\partial_y \varphi)\left(x, \sqrt{y^2 - t^2} \cdot \sin s\right) \sin s \,ds\\
  & + \frac{y^2}{y^2-t^2} \int_0^{\pi/2} 
  (\partial^2_{y} \varphi)\left(x, \sqrt{y^2 - t^2} \cdot \sin s\right) (\sin s)^2\, ds,
\end{align*}
\begin{align*}
  &\pi \partial^2_{t} w = \frac{-y^2}{(y^2-t^2)^{3/2}} \int_0^{\pi/2} 
  (\partial_y \varphi)\left(x, \sqrt{y^2 - t^2} \cdot \sin s\right) \sin s \,ds\\
  & + \frac{t^2}{y^2-t^2} \int_0^{\pi/2} 
  (\partial^2_{y} \varphi)\left(x, \sqrt{y^2 - t^2} \cdot \sin s\right) (\sin s)^2\, ds.
\end{align*}
Whence
\begin{align*}
  &\pi (\partial^2_{t} - \partial^2_{y}) w = 
  \frac{-1}{\sqrt{y^2-t^2}} \int_0^{\pi/2} 
  (\partial_y \varphi)\left(x, \sqrt{y^2 - t^2} \cdot \sin s\right) \sin s \,ds\\
  & - \int_0^{\pi/2} 
  (\partial^2_{y} \varphi)\left(x, \sqrt{y^2 - t^2} \cdot \sin s\right) (\sin s)^2\, ds.
\end{align*}
The first term on the right hand side equals
\begin{align*}
  &\frac{1}{\sqrt{y^2-t^2}} \int_0^{\pi/2} 
  (\partial_y \varphi)\left(x, \sqrt{y^2 - t^2} \cdot \sin s\right) 
  \frac{d}{ds}(\cos s) \,ds\\
  &=-\frac{(\partial_y\varphi)(x, 0)}{\sqrt{y^2-t^2}}
  - \int_0^{\pi/2} 
  (\partial^2_{y} \varphi)\left(x, \sqrt{y^2 - t^2} \cdot \sin s\right) (\cos s)^2 \,ds.
\end{align*}
Now applying the second relation in~\re{vphiprobl}
we obtain
\begin{align*}
  &\pi (\partial^2_{t} - \partial^2_{y}) w = 
  -\int_0^{\pi/2} 
  (\partial^2_{y} \varphi)
  \left(x, \sqrt{y^2 - t^2} \cdot \sin s\right)\, ds.
\end{align*}
Adding this to
\begin{align*}
  &-\pi\Delta_x w = 
  -\int_0^{\pi/2} 
  (\Delta_x \varphi)\left(x, \sqrt{y^2 - t^2} \cdot \sin s\right)\, ds,
\end{align*}
we obtain~\re{w1} in view of the first relation 
in~\re{vphiprobl}.
\end{proof}

It can be easily seen from the definition~\re{wdef} that 
$w$, $\partial_x w$ have continuous extensions on 
the set $\{y\geqslant|t|\}$.
This is also true for
$\partial_y w$, $\partial_t w$, which follows from 
the expressions for these derivatives 
given in the beginning of the previous proof and
from the second condition in~\re{vphiprobl}.
From now on we denote by $w$, $\partial_{x,y,t} w$
these continuous extensions defined on $\{y\geqslant|t|\}$.
We have
\begin{equation}\label{w22}
  (\partial_y w \pm \partial_t w)\big|_{t = \pm y} = 0. 
\end{equation}

\begin{lemma}\label{lemmaw}
  Suppose $\varphi$ is the solution of 
  the Cauchy problem~\re{helmholtz}, \re{helmholtz_cauchy}.
  Then for $w(x,y,t)$ defined by~\re{wdef} we have
\begin{equation}
  w, \partial_{x,y} w \to 0, \quad h\to 0, \quad \textrm{if } 
  |x| > y\geqslant |t|;
  \label{wlim}
\end{equation}
the convergence is uniform on every compact set in
$\{|x|>y\geqslant|t|\}$.
\end{lemma}
\begin{proof}
For $w$ and $\partial_x w$ assertion~\re{wlim} follows 
from~\re{vphilim} and formula~\re{wdef}.
To estimate $\partial_y w$, observe that due to
$\partial_y\varphi|_{y=0}=0$ we have
\begin{align*}
  \left|(\partial_y \varphi)\left(x, \sqrt{y^2 - t^2} \cdot \sin s\right)\right|
  \leqslant \sqrt{y^2-t^2}
  \max_{\tau\in[0,\sqrt{y^2-t^2}]}|(\partial^2_{y}\varphi)(x,\tau)|.
\end{align*}
Now applying~\re{vphilim} and formula for $\partial_y w$
given in the beginning of the proof 
of Lemma~\ref{wave_lemma}, we obtain~\re{wlim} for $\partial_y w$.
\end{proof}

Note that the condition $|x| > y\geqslant|t|$ in~\re{wlim} can be
weakened to $\sqrt{x^2+t^2} > y\geqslant |t|$ (although we will not use
this generalization).
Note also that if $\sqrt{x^2+t^2} < y$, then generally $w$
grows as $h\to 0$.

\section{Determination of the solution $u$} \label{determ}
In this section we will obtain our main result,
that is, formula~\re{formula}, which relates
$u(x^*,y^*,t^*)$ to
the Cauchy data on the set~\re{Cauchy_set}.

\begin{theorem}\label{T}
Under the conditions indicated in sec.~\ref{intro},
for any solution $u\in C^2(\overline\Omega\times{\mathbb R})$
of the wave equation~\re{waveq} we have
\begin{align}
  &u(x^*,y^*,t^*) = \notag\\
  &= \frac{1}{2} \sum_\pm u\left(x^*,Y(x^*), T_\pm(x^*)\right) 
  +\lim_{h\to 0} 
  \int_S d\sigma_{x,y} \int_{T_-(x)}^{T_+(x)} (u\, \partial_\nu w^* - \partial_\nu u\cdot w^* )\, dt,
    \label{formula}
\end{align}
where $w^*(x,y,t) = w(x-x^*,y-y^*,t-t^*)$,
the function $w$ being defined in Lemma~\ref{lemmaw},
$d\sigma$ is the surface measure on $\partial\Omega$.
\end{theorem}

From the remarks given in 
the end of sec.~\ref{spec_wave},
it follows that $w^*$ grows in the cone $K$ as $h\to 0$.
Hence the integrand on the right hand side of~\re{formula}
grows, which means that
if we plug there arbitrary smooth functions 
instead of the Cauchy data $u$, $\partial_\nu u$,
the limit of the integral generally does not exist.
So if the Cauchy data are given with some error (which is always the case in practice), one should approximate this limit
by the corresponding integral computed
for some positive $h$.
In fact, the optimal value of $h$ depends on the accuracy of 
the Cauchy data.
This issue arises in other problems that require a regularization including the analytic continuation problem.

\begin{proof}[Proof of Theorem~\ref{T}]
We suppose that the coordinates are chosen in such a way that
$(x^*,y^*,t^*) = (0,0,0)$.
In this case we have $T_\pm(x) = \pm Y(x)$.
Our condition $Y(x)>y^*$, $(x,y)\in \overline S$ 
imposed on $S$ in sec.~\ref{intro}
 now reads $Y(x)>0$, $(x,y)\in \overline S$.
First we prove~\re{formula} assuming that $Y(x)>0$ for all
$x\in{\mathbb R}^{n-1}$. 
After that we will eliminate this restriction.

Note that due to the growth condition in~\re{subgraph},
the set $K_\Omega = K\cap\overline\Omega$ is compact.
Hence there exists a compactly supported $C^\infty$-smooth function $\chi(x,y)$
in ${\mathbb R}^n$ such that $\chi=1$ in some neighborhood of $K_\Omega$.
Pick a number $R$ such that
the projection of ${\rm supp}\chi$
on the hyperplane $(x_1,\ldots,x_{n-1})$
is contained in the ball $\{|x|<R\}$.
Put
\[
  V = \{(x,y,t)\,|\,\, |x|<R,\, |t|<y<Y(x)\} \subset \Omega\times{\mathbb R}.
\]
The set $V$ is a bounded Lipschitz domain in ${\mathbb R}^{n+1}$.
Indeed, the diffeomorphism
\[
  (x,y,t) \mapsto (x, y/Y(x), t/Y(x))
\]
is well-defined in a neighborhood of $\overline V$, since
$Y(x)$ is separated from zero for bounded $x$.
It remains to observe that the specified diffeomorphism
maps $V$ to the Cartesian product $\{x\,|\,\,|x|<R\}\times \{(y,t)\,|\,\,|t|<y<1\}$ of Lipschitz domains.

Put $\tilde u = \chi u$.
We have
\begin{align}
  &\int_V \left[ w\cdot 
    (\partial^2_{t} - \Delta) \tilde u -
  \tilde u\cdot (\partial^2_{t} - \Delta) w\right] dx dy dt 
  \notag\\
  &= \int_{\partial V} \left[ (\tilde u\, \partial_x w
    -w\, \partial_x \tilde u) \nu_x + (\tilde u\, \partial_y w
    -w\, \partial_y \tilde u) \nu_y + (-\tilde u\, \partial_t w 
    +w\, \partial_t \tilde u) \nu_t \right] d\gamma. \label{surfint}
\end{align}
Here $\nu=(\nu_x,\nu_y,\nu_t)$
is the outward unit normal to $\partial V$,
$\nu_x=(\nu_{x_1},\ldots,\nu_{x_{n-1}})$,
$d\gamma$ is the surface measure on $\partial V$.
This formula is applicable provided the following conditions
are satisfied:
{\it i})~the functions $\tilde u$ and $w$ are $C^2$-smooth 
in $V$;
{\it ii})~the integral on the left hand side is absolutely convergent;
{\it iii})~the functions $\tilde u$, $w$ and their first order derivatives
have continuous extensions from $V$ to the closure $\overline V$
(the right hand side involves the values of 
these extensions on $\partial V$).
The condition ({\it i}) is obviously satisfied.
The condition ({\it ii}) follows from~\re{w1} and the fact that
$w$ is bounded in $V$, while 
$\tilde u$ is $C^2$-smooth in $\overline V$.
The last condition ({\it iii}) is obvious for $\tilde u$;
due to the remark given right after 
the proof of Lemma~\ref{wave_lemma},
({\it iii}) is satisfied for $w$.

Define the set
$\Gamma$ as the intersection of $\partial V$ with $\partial\Omega\times{\mathbb R}$,
while the sets $\Gamma^\pm$ are defined as the intersections of $\partial V$
with the hyperplanes $\{\pm t = y\}$.
For any point $(x,y,t)$ from $\partial V \setminus (\Gamma\cup \Gamma^+\cup \Gamma^-)$ we have $|x|=R$.
Due to our choice of $R$, in such points we have $\chi=0$,
and so $\tilde u=0$.
Therefore, the integral over $\partial V$ in~\re{surfint}
equals the sum of the corresponding integrals
over $\Gamma$, $\Gamma^+$, $\Gamma^-$.
The integral over $\Gamma^+$ equals
\begin{align}
  &\int_{\Gamma^+} \left[ (\tilde u\, \partial_y w
    -w\, \partial_y \tilde u) \nu_y + (-\tilde u\, \partial_t w 
    +w\, \partial_t \tilde u) \nu_t \right] d\gamma \notag\\
  &= \frac{1}{\sqrt{2}}\int_{\Gamma^+} \left[ 
    w\, (\partial_y \tilde u + \partial_t \tilde u)
    -\tilde u\, (\partial_y w + \partial_t w)
     \right]\, d\gamma \label{S+int}
\end{align}
(on $\Gamma^+$ we have $\nu_x=0$, $\nu_t = -\nu_y = 1/\sqrt{2}$).
Due to~\re{w22}, we have
$\partial_y w + \partial_t w = 0$ on $\Gamma^+$.
Now taking into account the equality 
$w(x,y,\pm y) = \varphi(x,0)/2$, which follows from the definition~\re{wdef},
previously obtained integral equals
\begin{align*}
  &\frac{1}{2}
  \int_{|x|<R} dx\, \varphi(x,0)
  \int_{0}^{Y(x)} (\partial_y \tilde u + \partial_t \tilde u)(x,y,y)\, dy \\
  &= \frac{1}{2}\int_{|x|<R} dx\, \varphi(x,0) 
  \int_{0}^{Y(x)} \partial_y(\tilde u(x,y,y))\, dy \\
  &= \frac{1}{2} \int_{|x|<R} \varphi(x,0)
  \left[ \tilde u(x, Y(x), Y(x)) - \tilde u(x, 0, 0) \right] dx.
\end{align*}
In the same way, we obtain that the integral~\re{S+int},
in which $\Gamma^+$ is replaced by $\Gamma^-$, equals 
\[
  \frac{1}{2} \int_{|x|<R} \varphi(x,0)
  \left[ \tilde u(x, Y(x), -Y(x)) - \tilde u(x, 0, 0) \right] dx.
\]
We conclude that the integral on the right hand side
of~\re{surfint} equals
\begin{align}
  &\int_{|x|<R} \varphi(x,0)
  \left[ \frac{1}{2} \tilde u(x, Y(x), Y(x)) + 
    \frac{1}{2} \tilde u(x, Y(x), -Y(x)) - \tilde u(x, 0, 0) \right] dx \notag\\
  &+ \int_\Gamma (\tilde u\, \partial_\nu w - \partial_\nu \tilde u \cdot w)\, 
  d\gamma. \label{rhslim}
\end{align}
Now we pass to a limit $h\to 0$ in~\re{surfint}.
The left hand side tends to zero.
Indeed, the second term of the integrand vanishes 
in view of~\re{w1}.
The function $(\partial^2_{t} - \Delta) \tilde u$
is nonzero only if 
$\partial_{x,y}\chi\ne 0$. 
However, the corresponding set of points $(x,y)$
is separated from $K_\Omega$ due to our choice of $\chi$.
In combination with~\re{wlim}, this gives
that the first term of the integrand tends to zero 
as $h\to 0$.

As it was previously shown, 
the right hand side of~\re{surfint}
equals the expression~\re{rhslim}.
Due to the first equality in~\re{helmholtz_cauchy},
$\varphi(x,0)$ is the Gaussian distribution in ${\mathbb R}^{n-1}$,
which tends to $\delta(x)$ as $h\to 0$.
Thus the first integral in~\re{rhslim}
tends to
\begin{align*}
  \frac{1}{2} \sum_\pm \tilde u(0, Y(0), \pm Y(0))
  - \tilde u(0, 0, 0)
    = \frac{1}{2} \sum_\pm u(0, Y(0), \pm Y(0))
    - u(0, 0, 0)
\end{align*}
as $h\to 0$
(we used the fact that $\chi|_{K_\Omega} = 1$).
It follows from the definition of the hypersurface $\Gamma$,
that the integral over $\Gamma$ in~\re{rhslim} equals
\[
  \int_{S_R} d\sigma_{x,y} \int_{-Y(x)}^{Y(x)} (\tilde u\, \partial_\nu w - 
  \partial_\nu \tilde u\cdot w)\, dt,
\]
where $S_R = \{(x,y) \,|\, |x|<R,\, y=Y(x)\} \subset\partial\Omega$.
Thus passing to a limit $h\to 0$ in~\re{surfint} yields
\begin{align*}
  u(0,0,0) = 
  \frac{1}{2} \sum_\pm u\left(0, Y(0), \pm Y(0)\right) 
  +\lim_{h\to 0} 
  \int_{S_R} d\sigma_{x,y} \int_{-Y(x)}^{Y(x)} (\tilde u\, \partial_\nu w - 
  \partial_\nu \tilde u\cdot w)\, dt.
\end{align*}
The integral over $S_R$ can be replaced by that
over $S$. This follows from~\re{wlim} and from 
the fact that the set
$S_R\setminus S$ is separated from $K_\Omega$.
Analogously, $\tilde u$ can be replaced by $u$,
since the set of points, in which $\chi\ne 1$,
is separated from $K_\Omega$.
Thus~\re{formula} is proved in case $Y(x)>0$, $x\in{\mathbb R}^{n-1}$.
Next we turn to the general case when $Y(x)$ is assumed to be positive only if $(x,y)\in \overline S$.

Let $X$ be the projection of the closure $\overline S$
on the hyperplane $(x_1,\ldots,x_{n-1})$.
Thus $Y(x)>0$ for $x\in X$, which implies that
$Y$ is positive in some neighborhood of $X$.
Hence there exists
a $C^\infty$-smooth function $\widetilde Y$ in ${\mathbb R}^{n-1}$ 
satisfying
$\widetilde Y|_X = Y|_X$ and
$0 < \widetilde Y(x) \leqslant C$ for $x\in {\mathbb R}^{n-1}$.
Now we introduce the domain
\[
  \widetilde\Omega = \{(x,y)\,|\,\, x\in {\mathbb R}^{n-1}, \, -\infty<y<\widetilde Y(x)\}.
\]
Despite the fact that $\widetilde Y$ is everywhere positive, 
the preceding proof of~\re{formula} can not be applied directly to
the domain $\widetilde\Omega$, since generally the solution $u$ is not defined in $\widetilde\Omega\times{\mathbb R}$.
However, this issue is handled by a proper choice of the function $\chi$ introduced in the beginning of the proof.
It is possible to take $\chi$ such that
$\chi=1$ in some neighborhood of $K_\Omega$ and $\chi(x,y) = 0$ whenever $x\notin X$,
since the projection of $K_\Omega$ 
on the hyperplane $(x_1,\ldots,x_{n-1})$ is contained in the interior of $X$.
Then the product $\tilde u = \chi u$ is well-defined for 
$x\in X$, $y<\widetilde Y(x)=Y(x)$, $t\in{\mathbb R}$, 
and can be smoothly continued by zero to the remainder of $\widetilde\Omega\times{\mathbb R}$. After that the preceding proof of~\re{formula} goes through.
\end{proof}

The proof of the preceding theorem does not require the solution $u$
to be defined everywhere in $\Omega\times{\mathbb R}$.
Suppose that the wave equation~\re{waveq} is satisfied
for $t\in{\mathbb R}$, and $(x,y)\in\Omega\setminus\omega$,
where $\omega$ is some relatively closed subset of $\Omega$.
Then formula~\re{formula} holds true provided the cone
$K$ is separated from $\omega$ (which implies, in particular, that 
$(x^*,y^*)\in\Omega\setminus\omega$).
The function $\chi$ in the proof
should be chosen so that it satisfies $\chi = 0$
in a neighborhood of $\omega$.
Then the product $\chi u$ can be smoothly continued by zero
to $\omega\times{\mathbb R}$ and thus it
 can be viewed as
a function defined everywhere in $\Omega\times{\mathbb R}$.
After that formula~\re{formula} is derived 
by the same argument.

\end{document}